\documentclass[12pt]{article}
\usepackage[utf8]{inputenc}
\usepackage[english]{babel}
\usepackage[T1]{fontenc}
\usepackage{lmodern}
\usepackage{xfrac}
\usepackage{marginnote}
\usepackage{graphicx}
\usepackage{caption}
\usepackage{subcaption}
\usepackage{amsmath}

\usepackage{mathrsfs}
\usepackage{mathtools}
\usepackage{amsthm,amssymb}
\usepackage[mathlines]{lineno}
\usepackage{enumerate}
\usepackage{enumitem}
\usepackage[colorlinks=true,citecolor=black,linkcolor=black,urlcolor=blue]{hyperref}
\usepackage[top=.9in, bottom=.9in, left=.5in , right=.5in]{geometry}
\usepackage{comment}
\usepackage{float}
\mathtoolsset{showonlyrefs}
\newcommand{\E}[2][]{\ensuremath{\mathbb{E}_{#1} \left[#2 \right]}}
\newcommand{\Prob}[2][]{\ensuremath{\mathbb{P}_{#1} \left(#2 \right)}}
\newcommand{\dd}{\mathrm d}

\newcommand{\N}{\mathbb{N}}

\def\supp{{\rm Supp}}

\providecommand{\mayberemove}[1]{{\color{green}{}}}

\newcommand{\R}{\mathbb{R}}

\newcommand{\mass}[1]{\left\langle m \mathbf{1}_{m \geq 1}, #1 \right\rangle}

\newcommand{\ME}{\mathcal{M}_{+}(E)}
\newcommand\numberthis{\addtocounter{equation}{1}\tag{\theequation}}

\newcommand{\cstone}{c'}

\newcommand{\eps}{\varepsilon}

\newtheorem{thm}{Theorem}[section]
\newtheorem{lemma}[thm]{Lemma}

\newtheorem{cor}[thm]{Corollary}
\newtheorem{assumption}[thm]{Assumption}
\theoremstyle{definition}
\newtheorem{rmq}{Remark}[section]
\newtheorem{example}[rmq]{Example}

\newtheorem{defn}{Definition}[section]

\begin{document}
\title{Gelation in cluster coagulation processes}
\author{L. Andreis\footnote{Politecnico di Milano, Dipartimento di matematica,
Piazza Leonardo da Vinci, 32,
20133 Milano.}, T. Iyer\footnote{Weierstrass Institute for Applied Analysis and Stochastics, Mohrenstrasse 39, 10117 Berlin, Germany.},  E. Magnanini\footnotemark[2].}
\maketitle
\abstract{We consider the problem of gelation in the cluster coagulation model introduced by Norris [\textit{Comm. Math. Phys.}, 209(2):407-435 (2000)], where pairs of clusters of types $(x,y)$ taking values in a measure space $E$, merge to form a new particle of type $z\in E$ according to a transition kernel $K(x,y, \mathrm{d} z)$. 
This model possesses enough generality to accommodate inhomogeneities in the evolution of clusters, including variations in their shape or spatial distribution.
We derive general, sufficient criteria for stochastic gelation in this model. As particular cases, we extend results related to the classical Marcus--Lushnikov coagulation process, showing that reasonable `homogenous' coagulation processes with exponent $\gamma>1$ yield gelation; and also, coagulation processes with kernel $\bar{K}(m,n)~\geq~(m \wedge n) \log{(m \wedge n)}^{3 +\epsilon}$ for $\epsilon>0$.\\
}
\section{Introduction} \label{intro}
Models of coagulation arise widely in many scientific models, in areas ranging from physical chemistry (in the formation of polymers), 
to astrophysics (in modelling the formation of galaxies). A classical model for coagulation involves collections of particles (which we refer to later on, formally, as `clusters'), each attributed with a mass. Pairs of particles coagulate to form new particles at a rate $\bar{K}(x,y)$, 
where $x$ and $y$ are \emph{masses} of the particles. This stochastic coagulation model is known as the \emph{Marcus--Lushnikov process}~\cite{Marcus68, Gil72, Lushnikov78}.
The limiting behaviour of the particle masses as the number of particles tends to infinity in the Marcus--Lushnikov model is generally expected to be encoded by a set of infinitely many differential equations (or measure-valued differential equations) known as the \emph{Smoluchowski} or \emph{Flory} equations. 

Particular cases of the Marcus--Lushnikov model are closely related to other stochastic models: the case $\bar{K}(x,y) = 1$ corresponds to the Kingman's coalescent~\cite{kingman-coal}, the case $\bar{K}(x,y) = x+y$ has multiple interpretations, including being related to Aldous' continuum random tree \cite{aldous-crt1, aldous-pitman-98-tree,  bertoin-eternal}, 
whilst the case $\bar{K}(x,y) = xy$ is closely related to the Erd\H{o}s--R\'{e}nyi random graph (see, for example, \cite{janson-birth-of-the-giant,aldous-mult-entrance, aldous-multiplicative-crit-window, AnKoPa21}). We refer the reader to the review paper~\cite{aldous99} for a more general overview, although remark that there has been a lot of progress made over the last 25 years. We also note that in this paper, we are interested in models of \emph{pure coagulation}; whilst a lot of work in the literature also allows for the \emph{fragmentation} of particles. 

A natural question of interest related to coagulation processes is whether or not at some time $t > 0$ there is the formation of \emph{macroscopic} or \emph{giant} particles (\emph{gels}), i.e. particles whose masses are on a significantly larger scale than the ones of the initial particles in the system. Motivated by the application to polymer chemistry, this is known as \emph{gelation}. Gelation is generally defined as whether a solution of the Smoluchowski (or Flory) equation fails to `conserve mass', which means, intuitively, that mass is lost to `infinite-mass' particles. This is closely linked to the appearance of `large particles' in the Marcus--Lushnikov model as long as one knows that the trajectories of the process concentrate on solution(s) of such an equation. 
This fact was first used by Jeon to prove existence of gelling solutions to the Smoluchowski equation~\cite{jeon98} whenever there exists $\alpha > 1/2$ such that $\bar{K}(x, y) \geq (xy)^{\alpha}$ (see also~\cite{rezakhanlou2013} for a simplified exposition). An alternative proof of gelation using analytic tools was provided by Escobedo, Mischler and Perthame~\cite{escobedo-mischler-perthame-02} (see also~\cite{escobedo-et-al-mass-cons}), whilst Laurencot~\cite{Lau14} improved this to show that gelation occurs whenever $\bar{K}(x,y) \geq \sqrt{xy}(\log(x)\log(y))^{1+\epsilon}$, for $\epsilon > 0$. However, a longstanding scientific conjecture that has not previously been proven rigorously states that if $\bar{K}$ is \emph{homogenous} with exponent $\gamma > 1$ (i.e., $\bar{K}(cx, cy) = c^{\gamma} \bar{K}(x,y)$ for $c > 0$) then gelation occurs (see, e.g. \cite{aldous99, eibeck-wagner-01, wagner-explosive-phenomena}). 

Although an extensive literature is devoted to the classical Marcus--Lushnikov process, relatively less is known about variants of this model incorporating inhomogenieties of particles, for example, their shape, velocity, or location in space. A framework introduced by Norris in~\cite{norris-cluster-coag}, called the \emph{cluster coagulation model}, allows one to incorporate these features in a rather general way. In~\cite{norris-cluster-coag} a weak law of large numbers  has also been proved, recently extended to weaker assumptions in~\cite{AnIyMa24}. See also~\cite{Norris} for a variant that incorporates diffusion of clusters. However, apart from a particular special case of the model~\cite{heydecker2019bilinear}, general criteria for gelation in this model that incorporate information about these inhomogeneities are lacking, despite being of interest from both the perspective of applications in physics, and mathematically. 

It should be noted that whilst this model is rather general, and can incorporate spatial characteristics in the clusters, it lacks an important feature: \emph{movement of particles in space} independent of coagulation events. There are a number of results related to other models incorporating the movement of particles as Brownian motions in space~\cite{hammond-rezakhanlou-07, hammond-rezkhanlou-moment-bounds-07, yaghouti-rezkhanlou-hammond-09, rezakhanlou-14-pointwise-bounds}, or as particles jumping across two sites~\cite{Schultze-wagner-06, wagner-post-gelation}; but we are not aware of any general results concerning criteria for gelation in these models, apart from the interesting~\emph{induced gelation} effect in a particular case of the two-site model~\cite{eibeck-wagner-01}. 
Recently, in~\cite{AnKoLaPa23}, the question of gelation in a spatial coagulation model has been approached using a different approach, with Poisson point processes and large deviations. 

\subsection{Overview on our contribution}
In this paper we make a contribution to the aforementioned gap in the literature by providing a general sufficient criterion for gelation in the cluster coagulation model. 
\begin{enumerate}
\item In Theorem~\ref{thm:simp-gel}, we provide sufficient criteria for gelation in the general setting of the cluster coagulation model. These criteria may be of interest in applied settings, where the kernels may take into account many features of a cluster, not only the mass. 
\item \label{item:class-improve} As a consequence, in Corollary~\ref{cor:classical-gelation} we provide improved criteria for gelation in the classical Marcus--Lushnikov process. We show, as a particular case, that reasonable `homogenous' coagulation processes with exponent $\gamma > 1$ yield gelation, thus, providing, as far as we now, the first rigorous proof of the aforementioned generally accepted scientific principle. 
We also show, that, if for $\epsilon > 0$, $\bar{K}(x,y) \geq (x \wedge y) (\log{(x \wedge y)})^{3 +\epsilon}$, gelation occurs - a condition that depends only on the size of the minimum of merging clusters.  
\end{enumerate}

To prove the gelation statements, we apply techniques introduced by Jeon~\cite{jeon98} (see also~\cite{rezakhanlou2013}), exploiting the Markovian dynamics and the generator of the process to bound the expected value of the random time when a positive fraction of the mass of the system is made up of `large' clusters. In the general spatial setting, we use a coarse-graining procedure, partitioning the space into regions where the interaction rates between clusters have sufficient lower bounds. 
The summability condition in Assumption~\ref{ass:gelation} ensures large clusters form `quickly' enough for gelation to arise. When the state space also incorporates, for example, the position of the clusters, this criterion guarantees gelation for kernels that are products of a function decreasing in the distance between clusters, and of a `gelling' kernel of the masses. 
The improved sufficient criteria for gelation in the non-spatial Marcus-Lushnikov model, outlined in Item~\ref{item:class-improve} above, are a result of finer lower bounds on the generator compared to the bounds in~\cite{jeon98, rezakhanlou2013}. 

The rest of the paper is structured as follows.
\begin{enumerate}
    \item In Section~\ref{intro-subsec-model} we introduce the cluster coagulation model and we define gelation in this context. In Section~\ref{example_CGP} we provide examples of natural models that fit into the framework of cluster coagulation processes.
    \item In Section~\ref{sec:main-results-statements} we state our results. In particular, in Section~\ref{sec:suff_crit} we state our sufficient criteria for gelation, the main theorem of this paper. In Corollary~\ref{cor:classical-gelation} we state how this criteria translate into the setting of classical, non-spatial, coagulation processes, improving the state of the art in this setting.  In Section~\ref{sec:coupling} we mention a connection with inhomogeneous graphs and how to derive some gelation criteria from them. Such a connection is based on a natural coupling, which we omit to prove here (we refer to an extended pre-print version of this work for it~\cite{AnIyMa23gel}). 
    \item Finally, Section~\ref{sec:main-results-proofs} deals with the proofs of the main results. 
\end{enumerate}

\subsection{Definition of the process and gelation}
\label{intro-subsec-model}
We consider the \emph{cluster coagulation process}, introduced by Norris in~\cite{norris-cluster-coag}, 
where \emph{clusters} are characterised by their \emph{type}, an element of a measurable space $(E, \mathcal{B})$. Associated with a cluster of type $x \in E$ is a \emph{mass function} $m: E \rightarrow (0, \infty)$. Another important quantity associated with the process is a \emph{coagulation kernel} $K: E \times E \times \mathcal{B} \rightarrow [0, \infty)$, which satisfies the following:
\begin{enumerate}
    \item for all $A \in \mathcal{B}$ $(x,y) \mapsto K(x,y, A)$ is measurable;
    \item for all $x, y \in E$ $K(x,y, \cdot)$ is a measure on $E$;
    \item \emph{symmetric:} for all $A \in \mathcal{B}, x,y \in E$ $K(x,y, A) = K(y,x, A)$;
    \item $K$ is \emph{finite:} for all $x, y  \in E$ $\bar{K}(x,y) := K(x,y, E) < \infty$;
    \item  $K$ \emph{preserves mass:} for all $x, y \in E$, 
    $m(z) = m(x) + m(y)$ for $K(x,y, \cdot)$-a.a. $z \in E$.
\end{enumerate}

Suppose that we begin with a configuration of clusters labelled by an index set $I$. 
Then, 
\begin{itemize}
    \item to each labelled pair of clusters, independently, with types $x, y \in E$ and $\bar{K}(x,y) > 0$ we associate an exponential clock (exponential random variable) with parameter $\bar{K}(x,y)$;
    \item upon the elapsure of the next exponential random variable in the process, corresponding to a labelled pair with types $x,y \in E$, say, the associated clusters are removed and replaced with a new labelled cluster with type $z \in E$, where $z$ is sampled from the probability measure 
    \begin{equation}\label{measz}
    \frac{K(x, y, \cdot)}{\bar{K}(x, y)}.
    \end{equation}
\end{itemize}
We refer the reader to Section \ref{example_CGP} for natural examples of coagulation models that fit into this framework.
Figure~\ref{coag_mech} illustrates the coagulation mechanism.
\begin{figure}[H]
\centering
\includegraphics[width=0.8\linewidth]{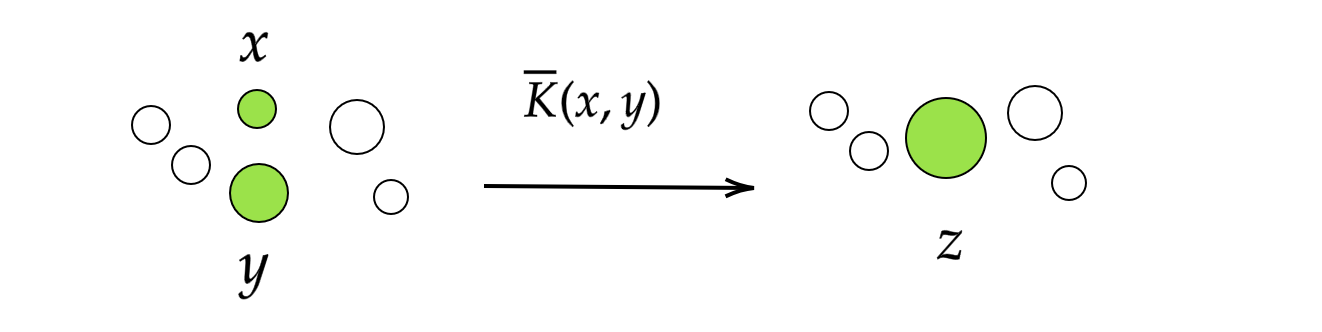}
 \caption{Two clusters of types $x$ and $y$ merge at rate $\bar{K}(x,y)$ into a particle of type $z$, sampled from the probability measure $K(x,y, \dd z)/\bar{K}(x,y)$, with mass $m(z)=m(x)+m(y)$. }  \label{coag_mech}
 \end{figure}

Throughout this paper we consider sequences of cluster coagulation processes, depending on a parameter $N \in \mathbb{N}$, that is for each process of the sequence we have a coagulation kernel that depends on the parameter (we denote it with $K_N$). One may consider this parameter $N$, up to random fluctuations, as the total \emph{initial mass} of the system. We analyse the process in limiting regimes as $N \to \infty$. For each $N \in \mathbb{N}$, we consider the process as a measure valued Markov process $(\mathbf{L}^{(N)}_{t})_{t \geq 0}$ on the space of \emph{finite point measures} on $E$, i.e., the space of positive, integer valued,  finite measures on $E$. The point measure $\mathbf{L}^{(N)}_{t}$ encodes the configuration of clusters at time $t$, so that for $A \subseteq E$, $a \in (0, \infty)$, $\mathbf{L}^{(N)}_{t}(A \cap m^{-1}([a, \infty)))$
denotes the random number of clusters of mass at least $a$ belonging to $A$ at time $t$. Suppose $\mathcal{M}_{+}(E)$ denotes the set of finite, positive measures on $E$.
Then, the infinitesimal generator $\mathcal{A}_{N}$ associated with the process is defined as follows: for any bounded measurable test function $F:\mathcal{M}_{+}(E) \rightarrow \mathbb{R}$, we have
\begin{equation} \label{eq:gen-def}
\mathcal{A}_{N} F(\xi)=  \frac{1}{2}\int_{E\times E\times E} \xi(\dd x) \left(\xi - \delta_{x}\right)(\dd y) K_{N}(x,y,\dd z) \left(F({\xi}^{(x,y)\to z}) - F(\xi)\right) \qquad \forall \xi\in \mathcal{M}_{+}(E),
\end{equation}
where ${\xi}^{(x,y)\to z} := \xi + \left(\delta_{z} -\delta_{x} - \delta_{y} \right)$. Note that, the measure $\xi$ in~\eqref{eq:gen-def} is always non-negative, since, we always assume the process is initiated by a finite point measure. Thus, for any $x \in E$, by assumption $\xi(\{x\}) \in \mathbb{N}_{0}$ is a non-negative integer, representing the number of clusters of type $x$. The measure ${\xi}^{(x,y)\to z}$, therefore, describes the configuration of the system after a coagulation involving the two clusters $x,y\in E$ merging to form a cluster $z$ with $m(z)= m(x) +m(y)$, for $K_{N}(x,y, \cdot)$-a.a. $z \in E$. 
\begin{rmq}
Abusing notation for brevity, if we write $\xi(\dd x)$ for $\xi(\{x\})$, the factor $\frac{1}{2}$ in front of the generator in~\eqref{eq:gen-def} ensures that the number of pairs that interact at rate $\bar{K}_{N}(x,y)$ is $\xi(\dd x)\xi(\dd y)$ if $x \neq y$ and $\binom{\xi(\dd x)(\xi(\dd x)-1)}{2}$ if $x=y$.  
\end{rmq}
Whilst we allow the kernel $K_{N}$ to depend on $N$, we are often motivated by scenarios where $K_{N} \equiv K$ for some fixed kernel $K$. 
In such cases, it is well known that, in order to observe non-trivial limiting behaviour of the process (for example in proving a law of large numbers) a time rescaling is needed. This allows one to counterbalance the increase in the number of interactions as the initial mass of clusters grows with $N$.
This motivates the definition of the following normalised process: 
\[
\bar{\mathbf{L}}^{(N)}_{t} := \mathbf{L}^{(N)}_{t/N}/N.
 \]
Note that the normalised process is still a Markov process on $\mathcal{M}_{+}(E)$ with generator $\tilde{\mathcal{A}}_{N}$, defined such that for $\tilde{\xi} \in \mathcal{M}_{+}(E)$ and any bounded measurable test function $F:\ME \rightarrow \mathbb{R}$, we have
\begin{linenomath}
\begin{align} \label{eq:gen-def-2}
\tilde{\mathcal{A}}_{N}F(\tilde{\xi}) = \frac{N}{2}\int_{E\times E\times E} \tilde{\xi}(\dd x) \left(\tilde{\xi} - \frac{\delta_{x}}{N}\right)(\dd y) K_N(x,y,\dd z)\left(F\left(\tilde{\xi} + \frac{\left(\delta_{z} -\delta_{x} - \delta_{y} \right)}{N}\right) - F(\tilde{\xi})\right).
\end{align}
\end{linenomath}
The scaling by $N$ comes from the following. First, to get the total rate at which clusters associated with $x$ and $y$ interact, we need to multiply each of the terms $\tilde{\xi}(\dd x)$ and $\left(\tilde{\xi} - \frac{\delta_{x}}{N}\right)(\dd y)$ by a factor of $N$. Then, note that the time re-scaling $t \mapsto t/N$ corresponds to `slowing down' each of the exponential clocks in the system by a factor of $N$. If $X$ is exponentially distributed with rate $K$, $N X$ is exponentially distributed with rate $K/N$. Therefore, we need to normalise the kernel by $K_{N} \mapsto K_{N}/N$. Combining these steps leads to the above. 
 Given a measure $\mu \in \mathcal{M}_{+}(E)$ and a measurable function $f: E \rightarrow \mathbb{R}$, we denote by \[\langle f, \mu \rangle := \int_{E} f(x) \mu(\dd x).\] 
At the level of the stochastic process, we denote by $\Prob[N]{\cdot}$ and $\E[N]{\cdot}$ probability distributions and expectations with regards to the trajectories of the process $(\bar{\mathbf{L}}^{(N)}_{t})_{t \geq 0}$ with generator $\tilde{\mathcal{A}}_{N}$ and possibly random initial condition $\bar{\mathbf{L}}^{(N)}_{0}$. To ensure the process is well defined, we assume throughout that $N \bar{\mathbf{L}}^{(N)}_{0}$ is, almost surely, a point measure. 
In addition, we introduce the following notation for the regular conditional distribution and expectations when the initial condition $\bar{\mathbf{L}}^{(N)}_{0}$ is given by a (deterministic) measure $\boldsymbol{\pi}^{(N)}\in \mathcal{M}_{+}(E)$ 
\begin{equation}\label{eq:prob_pi_N}
\Prob[N, \, \boldsymbol{\pi}^{(N)}]{\cdot} := \Prob[N]{\cdot \, | \, \bar{\mathbf{L}}^{(N)}_{0} = \boldsymbol{\pi}^{(N)}} \quad \text{and} \quad  \E[N, \,\boldsymbol{\pi}^{(N)}]{\cdot} := \E[N]{\cdot \, | \, \bar{\mathbf{L}}^{(N)}_{0} = \boldsymbol{\pi}^{(N)}}.
\end{equation}
\begin{rmq} \label{rem:finite-point-measure}
As with $\bar{\mathbf{L}}^{(N)}_{0}$, we always only consider initial configurations $\boldsymbol{\pi}^{(N)}$ such that $N \boldsymbol{\pi}^{(N)}$ is a finite point measure. 
\end{rmq}
For brevity, we generally refer to sequences of cluster coagulation processes $\big\{(\bar{\mathbf{L}}^{(N)}_t)_{t \geq 0}\big\}_{N \in \mathbb{N}}$ as a cluster coagulation process. 

Our main interest is the presence of \emph{gelation} in the cluster coagulation process, indicating the emergence of `large' clusters. The following is a minor reformulation of the definitions from Jeon~\cite[Definition~2]{jeon98}. First, it is helpful to define random times indicating the presence of these large cluster. Let $\psi: \mathbb{N} \rightarrow \mathbb{N}$ be given such that $\lim_{N\to\infty}\psi(N)=+\infty$, and let $\delta \in (0,1)$. For each $N \in \mathbb{N}$ we define the $(\psi,\delta)$\emph{-gelation time} $\tau_{N}(\psi, \delta)$ by
   \begin{equation} \label{eq:giant-time-coagulation}
 \tau_{N}(\psi, \delta) := \inf \left\{t \geq 0: \left\langle m \mathbf{1}_{m \geq \psi(N)}, \bar{\mathbf{L}}^{(N)}_t\right\rangle \geq  \delta\right\}.\end{equation}
\begin{defn}[Stochastic gelation and strong gelation] \label{defn:stoch-gelation}
Take a cluster coagulation process \newline $\{(\bar{\mathbf{L}}^{(N)}_t)_{t \geq 0}\}_{N\in\mathbb{N}}$:
\begin{enumerate}
    \item for a function $\psi: \mathbb{N} \rightarrow \mathbb{N}$ such that $\lim_{N\to\infty}\psi(N)=+\infty$ and $\delta > 0$, the \newline\emph{$(\psi, \delta)$-stochastic gelation time} of the cluster coagulation process $(\bar{\mathbf{L}}^{(N)})_{N \in \mathbb{N}}$ is defined by 
\begin{equation}\label{stoc_gtime}
T^{\psi, \delta}_{g} := \inf\left\{t \geq 0: \limsup_{N \to \infty} \Prob[N]{\tau_{N}(\psi, \delta) \leq t} > 0 \right\};
\end{equation} 
    \item we define the \emph{strong gelation time} of the cluster coagulation process by \[\inf\{t > 0: \exists \, 0 < \alpha, \delta \leq 1 \text{ such that } \limsup_{N \to \infty} \Prob[N]{\tau_{N}(\alpha N, \delta) \leq t} > 0\}.\]  
\end{enumerate}
If, for some $\psi: \mathbb{N} \rightarrow \mathbb{N}$ such that $\lim_{N\to\infty}\psi(N)=+\infty$ and $\delta > 0$ we have $T^{\psi, \delta}_{g} < \infty$ we say that \emph{stochastic gelation occurs}. If the strong gelation time is finite, we say that \emph{strong gelation occurs}.
\end{defn}

\begin{rmq}
    Gelation is also often defined in terms of deterministic trajectories. More precisely, under suitable assumption of convergence of the sequence of kernels $\{K_N\}_N$, one expects the trajectories of $(\bar{\mathbf{L}}^{(N)}_{t})_{t \geq 0}$ to concentrate on solutions of the so called `Smoluchowski' equation as $N \to \infty$, or better to its modification that includes the influence of gel, the so called `Flory' equation, see~\cite{norris-cluster-coag, AnIyMa24}. Suppose $(\bar{\mathbf{L}}^{*}_t)_{t \geq 0}$ is a trajectory solving this equation, with $\left\langle m, \bar{\mathbf{L}}^{*}_{0} \right\rangle < \infty$.  One says $(\bar{\mathbf{L}}^{*}_t)_{t \geq 0}$ is a `gelling solution' if for some $t, \eps > 0$, $\left\langle m, \bar{\mathbf{L}}^{*}_t \right\rangle \leq \left\langle m, \bar{\mathbf{L}}^{*}_{0} \right\rangle - \eps$. An equivalence principle by Jeon, which one may generalise to the cluster coagulation setting, shows that whenever $(\bar{\mathbf{L}}^{(N)}_{t})_{t \geq 0}$ displays this concentration, stochastic gelation is equivalent to the existence of gelling solutions to the associated equation (see~\cite[Theorem~5]{jeon98}).
\end{rmq}

\subsection{Examples of cluster coagulation processes} \label{example_CGP}
The cluster coagulation process is general enough to encompass a large number of examples, depending on particular choices of the space $E$. In the following examples we fix $K_{N} \equiv K$ for some kernel $K$. 
\begin{example}[Classical kernel] \label{ex:classical}
    If $E = (0, \infty)$, $K(x, y, \dd z) =\bar{K}(x,y) \delta_{x + y}$, for a continuous symmetric function $\bar{K}(x,y)$, and the mass function $m(x) \equiv x$, the above process corresponds to the classical Marcus--Lushnikov process.
\end{example}

\begin{example}[Historical Marcus--Lushnikov processes]
    One may extend $E$ to incorporate not just the masses of clusters, but their histories. Indeed, we can take $E$ to be 
    a space where clusters $x$ encode %
    not only their mass, but the history of coagulations (a binary tree embedded in time) leading to the formation of that particle (see~\cite{jacquot-10} for more details). For these processes, Jacquot in~\cite{jacquot-10} proved a weak law of large numbers for the trajectories $(\bar{\mathbf{L}}^{(N)}_{t})_{t\in [0,T)}$ when the kernel is a function only of the associated masses and it is bounded from above by a product of sublinear functions. 
\end{example}
\begin{example}[Toy spatial coagulation models] \label{ex:toy-spatial-coag-1}
    A large number of toy models that incorporate information about the locations of clusters in `space' fall into this framework. For example, we may take $E=\mathcal{S}\times (0, \infty)$ where $\mathcal{S}\subseteq \mathbb{R}^{d}$; in this case an element $x$ of $E$ coincides with a pair $(p,n)$, $p\in\mathcal{S}, n\in (0, \infty)$ and we interpret $p$ as the \emph{location} of a cluster, and $n := m(p, n)$ as its mass. We may, then, assume that after a coagulation between clusters $x = (p,n), y = (s, o)$, the new cluster is placed at a new location, given by a measurable function of the original clusters, for example, the \emph{centre of mass} $\frac{np + os}{n+o}$.  Thus, in this case $K((p,n), (s, o), \cdot) = \bar{K}((p,n), (s, o)) \delta_{\frac{np + os}{n+o}, n+o}$. Another alternative would be a model in which the new particle occupies one of the locations of the previous clusters with probability proportional to their mass, so that 
    $K((p,n), (s, o), \cdot) =\bar{K}((p,n), (s, o))\left(\frac{n}{n+o} \delta_{p, n+o} + \frac{o}{n+o} \delta_{s, n+o}\right)$ (this is the way in which  the `collision operator' is defined in the model of coagulating Brownian clusters of~\cite{hammond-rezakhanlou-07}).
\end{example}

\begin{example}[Bilinear coagulation processes] \label{ex:bilinear}
    In the case that $E = [0, \infty)^{d}$, $A \in [0, \infty)^{d \times d}$ is a symmetric matrix with non-negative entries and $K(x, y, \dd z) = (x^{T} A y) \delta_{x+ y}$, this model corresponds to the bilinear coagulation model studied in~\cite{heydecker2019bilinear}. In that paper, the authors prove a weak law of large numbers for the particle system, showing that the trajectories converge to the unique solution of the Flory equation, and characterise explicitly the `gelling time', 
    by using comparisons between this process and \emph{inhomogeneous random graph processes}.\footnote{Actually, the model studied in~\cite{heydecker2019bilinear} is slightly more general, in that clusters $x$ belong to a metric space $S$, and $\bar{K}(x,y) = \pi(x)^{T} A \pi(y)$, where $\pi: \mathcal{S} \rightarrow \R^{d}$ is a continuous function. Clusters $x$ may also change values according to a kernel $J$ on $S$, in such a way that $\pi(x)$ is preserved.} 
\end{example}

\section{Gelation in the coagulation process}\label{sec:main-results-statements}
In this section, we state general sufficient conditions for stochastic gelation in the cluster coagulation model. As this model is rather general, the conditions required are more technical than conditions for the classical Marcus--Lushnikov process. The main motivation for these results is that, in applications to non-equilibrium processes inhomogenieties in the space $E$ (corresponding to, for example, locations in space, the `types' of cluster, or their velocities) may play a major role in whether or not gelation occurs. 

\subsection{Main result: sufficient criteria for stochastic gelation}\label{sec:suff_crit}
In Assumption~\ref{ass:gelation}, we incorporate the inhomogeneities into the gelation criterion via a coarse-graining procedure. In other words, we assume that, for each $N$, we can `partition' the space $E$ in such a manner that we have sufficient lower bounds on the rate at which clusters belonging to a common partition interact. For example, if $E$ is a metric space, we might partition the space into balls of a fixed radius, thus grouping together clusters that are `close', or of a similar `type'. Note that, in our assumptions, these partitions are allowed to grow with $N$, allowing for a `finer' approximation as $N$ is made larger.

The techniques we use extend those previously developed for the Marcus--Lushnikov process by Jeon~\cite{jeon98} (see also Rezakhanlou~\cite{rezakhanlou2013}).
\begin{assumption}[Assumptions on the kernel] \label{ass:gelation}
Suppose that $\{K_N(x,y, \dd z)\}_{N\in \mathbb{N}}$ is a sequence of kernels associated with a cluster coagulation process. 
For functions $\xi, \psi: \mathbb{N} \rightarrow \mathbb{N}$, with $\lim_{N \to \infty} \psi(N) = \infty$ we assume the following. 
\begin{enumerate}
    \item \label{item:partition_lower_bound} For each $j \leq \log_2(\psi(N))$ there exists a partition $\mathscr{P}^{(j)}_N$ of the set $m^{-1}([2^{j}, 2^{j+1})) \subseteq E$
    such that $\left|\mathscr{P}^{(j)}_N\right| \leq \xi(N)$. Moreover, with $\cstone_N(P,j):=\inf\{\bar{K}_{N}(x,y): x,y\in P\in \mathscr{P}^{(j)}_N\}$, we have
    \begin{equation} \label{eq:ass-1-gel}
     \cstone_N(P,j) > 0 \qquad \forall P\in \mathscr{P}^{(j)}_N. 
    \end{equation}

   \item\label{item:condition_diagonal} 
   There exists a sequence $(f_{j})_{j \in \mathbb{N}_{0}} \in (0, \infty)^{\mathbb{N}_{0}}$ such that $\sum_{j=0}^{\infty}f_{j}<\infty$, and
    \begin{equation} \label{eq:summability}
    \limsup_{N \to \infty} \sum_{j=0}^{\lceil\log_2(\psi(N))\rceil} \frac{2^j \phi_{N}(j)}{f_j^2} < \infty,
    \end{equation}
    where $\phi_{N}(j) := \sum_{P \in \mathscr{P}_{N}^{(j)}}\frac{1}{\cstone_N(P,j)} $.
    \item \label{item:condition_pidgeon} We have $\lim_{N \to \infty} \frac{\xi(N)}{N} = 0$.
\end{enumerate}
\end{assumption}
\begin{rmq}
The intuition behind Assumption~\ref{ass:gelation} can be summarised as follows.
\begin{enumerate}
\item Item~\ref{item:partition_lower_bound} guarantees that we can partition the space $E$ in such a way that a pair of clusters has strictly positive rate of interaction if the two clusters are close enough and have `comparable masses'.
\item
Item~\ref{item:condition_diagonal}, ensures that the lower bounds in~\eqref{eq:ass-1-gel} are sufficiently large to render the quantity in \eqref{eq:summability} finite. This is a technical assumption that we need to guarantee that the expectation of $\tau_{N}(\psi,\delta)$
is bounded from above by a finite value.
\item Item~\ref{item:condition_pidgeon} ensures the partitions are not too `fine'. Under this condition, the pigeonhole principle can be used to show that there are enough clusters in each element of the partition to guarantee that the process does not get `stuck'. 
\end{enumerate}
\end{rmq}

\begin{rmq}\label{equiv}
In Item~\ref{item:condition_diagonal} of Assumption~\ref{ass:gelation}, if the choices of partitions can be made in a manner independent of $N$, so that $\phi_{N}(\cdot) \equiv \phi(\cdot)$, then~\eqref{eq:summability} reduces to showing that, for some sequence $(f_{j})_{j \in \mathbb{N}_{0}} \in (0, \infty)^{\mathbb{N}_{0}}$ with $\sum_{j=0}^{\infty} f_{j} < \infty$, we have 
$\sum_{j=0}^{\infty} \frac{2^{j} \phi(j)}{f_{j}^2} < \infty$. 

This may be formulated in a more elegant way: there exists a sequence of positive values $\{f_i\}_{i\in\N}$ such that $\sum_{i\in\N}f_i<\infty$ and   
 $\sum_{i\in\N} 2^j\phi(j)/f_i^2  < \infty$, 
   if and only if
   $\sum_{i\in\N}(2^j\phi(j))^{1/3}<\infty.$  Indeed, on the one hand, if $\sum_{i\in \N} (2^j\phi(j))^{1/3}<\infty$, then it is enough to choose $f_j=(2^j\phi(j))^{1/3}$.
    For the other implication, by H\"{o}lder's inequality we see that $\sum_{j\in\N}(2^j\phi(j))^{1/3} \leq \left(\sum_{j\in \N}\frac{(2^j\phi(j))}{f^2_j}\right)^{\frac{1}{3}}\left(\sum_{i\in\N}f_i\right)^{\frac{2}{3}}$.
\end{rmq}
In the following theorem, recall that by Remark~\ref{rem:finite-point-measure} we only consider initial configurations $\boldsymbol{\pi}^{(N)}$ that are finite point processes.  Moreover, recall Definition~\ref{defn:stoch-gelation} and the definition of $\tau_N$ from~\eqref{eq:giant-time-coagulation}. 

\begin{thm} \label{thm:simp-gel}
 Suppose that $\left\{(\bar{\mathbf{L}}^{(N)}_t)_{t \geq 0}\right\}_{N \in \mathbb{N}}$ is a cluster coagulation process satisfying Assumption~\ref{ass:gelation}.
\begin{enumerate}
    \item \label{item:thme-simp-gel-1} 
    Suppose that $\{\boldsymbol{\pi}^{(N)}\}_{N \in \mathbb{N}}$ is a sequence of deterministic initial conditions such that, there exists $\eps > 0$ and  $\rho_{0} \in (0,1]$, such that for all $N \in \mathbb{N}$, $\mass{\boldsymbol{\pi}^{(N)}} > \eps$, and 
    \begin{equation} \label{eq:initial-condition}
    \frac{\left\langle m \mathbf{1}_{m \geq 1}, \boldsymbol{\pi}^{(N)} \right\rangle}{\mass{\boldsymbol{\pi}^{(N)}}} \geq \rho_{0}.
    \end{equation} 
    Then, there exists a function $\psi'$ 
    with $\lim_{N \to \infty} \psi'(N) = \infty$, such that, for any $\delta \in (0, \rho_{0}\eps)$ 
    \begin{equation} \label{eq:thm-main}
    \limsup_{N \to \infty} \E[N,\boldsymbol{\pi}^{(N)}]{\tau_N(\psi'(N),\delta)} < C,
    \end{equation} 
    for a  constant $C$, independent of $\boldsymbol{\pi}^{(N)}$ and $\delta$.  Moreover, if $\lim_{N \to \infty} \mass{\boldsymbol{\pi}^{(N)}} = \infty$, then \[\lim_{N \to \infty} \E[N,\boldsymbol{\pi}^{(N)}]{\tau_N(\psi'(N),\delta)} = 0.\]
\item \label{item:thme-simp-gel-2} 
Suppose that $\{\bar{\mathbf{L}}^{(N)}_{0}\}_{N \in \mathbb{N}}$ is sequence of random initial conditions such that 
for some $\eps>0$ and $\rho_0\in (0,1]$
\begin{equation} \label{eq:liminf-prob-gel}
\limsup_{N \to \infty} \Prob[N]{\mass{\bar{\mathbf{L}}^{(N)}_{0}} > \eps, \frac{\left\langle m \mathbf{1}_{m \geq 1}, \bar{\mathbf{L}}^{(N)}_{0} \right\rangle}{\mass{\bar{\mathbf{L}}^{(N)}_{0}}} \geq \rho_{0}} > 0. 
\end{equation}
Then stochastic gelation occurs in the process. 
\end{enumerate}
\end{thm}
\begin{rmq} \label{rem:mass-bound-below}
    In Theorem~\ref{thm:simp-gel}, the condition on $\mass{\boldsymbol{\pi}^{(N)}}$ ensures that there is enough mass bounded from below to form a gel. We can also prove stochastic gelation when the indicators $\mathbf{1}_{m \geq 1}$ may be replaced by any $\mathbf{1}_{m \geq c}$ for any $c > 0$, as long as Assumption~\ref{ass:gelation} is satisfied for the re-scaled function $\tilde{m} := m/c$ instead of $m$. Indeed, gelation with the re-scaled mass function $\tilde{m}$ also implies gelation with the original mass function $m$.
\end{rmq}

\begin{rmq}
In Theorem~\ref{thm:simp-gel}, we can think of $N$ as describing the size of the system. 
Item~\ref{item:thme-simp-gel-1} is a quenched result, holding for a sequence of deterministic initial conditions $\{\boldsymbol{\pi}^{(N)}\}_{N \in \mathbb{N}}$, whilst 
Item~\ref{item:thme-simp-gel-2} provides criteria for stochastic gelation given a sequence of random initial conditions. 
\end{rmq}
\begin{rmq}
 The second part of Item~\ref{item:thme-simp-gel-1} in Theorem~\ref{thm:simp-gel} proves that when $\lim_{N \to \infty} \mass{\boldsymbol{\pi}^{(N)}} = \infty$, and Assumptions~\ref{ass:gelation} is satisfied, we have the so called instantaneous gelation.  
\end{rmq}

The following corollary applies these conditions to the classical Marcus--Lushnikov process, i.e., the setting of Example~\ref{ex:classical}. In particular, it provides criteria for stochastic gelation that improves those appearing in~\cite[Corollary~1]{jeon98},~\cite[Theorem~1.3]{rezakhanlou2013} and~\cite[Proposition 3.15]{Lau14}. In this context, fix a kernel $K_{N} \equiv K$ (removing dependence on $N$). 
\begin{cor} \label{cor:classical-gelation} 
Let $\big\{(\bar{\mathbf{L}}^{(N)}_t)_{t \geq 0}\big\}_{N \in \mathbb{N}}$ be a sequence of Marcus-Lushnikov processes with kernel $K$, i.e., we are in the setting of Example~\ref{ex:classical}. 
Suppose that $\exists\,\eps>0$ and $\rho_0\in (0,1]$ such that~\eqref{eq:liminf-prob-gel} is satisfied for the sequence of initial conditions $\big\{\bar{\mathbf{L}}^{(N)}_0\}_{N \in \mathbb{N}}$.   
Moreover, with $\cstone(j):= \inf_{x,y \in [2^{j}, 2^{j+1})}\bar{K}(x,y)>0$, for all $j \in \mathbb{N}_{0}$, suppose that 
\begin{equation}\label{eq:class-gel-sum-cond}
    \sum_{j=0}^{\infty} \left(\frac{2^{j}}{\cstone(j)} \right)^{\frac 13}< \infty.
\end{equation}
Then stochastic gelation occurs.
     In particular, when ~\eqref{eq:liminf-prob-gel} is satisfied, stochastic gelation occurs if 
     one of the following two assumptions is satisfied by the kernel $\bar{K}$:
    \begin{enumerate}
        \item \label{item:no_space1} we have $\inf_{u\in [1, 2]} \bar{K}(1, u) > 0$ and for all $x, y > 0$ and all $c>0$,  $\bar{K}(cx, cy) = c^{\gamma} \bar{K}(x, y)$, with $\gamma > 1$;
        \item \label{item:no_space2}
        there exists $\epsilon > 0$ such that, for all $x, y\geq 1$, $\bar{K}(x, y) \geq 1+ (x \wedge y) \log{(x \wedge y)}^{3 + \epsilon}$.
    \end{enumerate}
\end{cor}
In the space-homogeneous setting of Corollary~\ref{cor:classical-gelation}, we have the following remarks.

\begin{rmq}
    Intriguingly, Equation~\eqref{eq:class-gel-sum-cond} shows that gelation may occur for kernels where $\bar{K}(x,y) = 0$ whenever $y/x > 2$ (or $y/x <1/2$). This criterion shows that interactions between particles of `similar size' is enough to ensure gelation, as long as it is strong enough (hence the condition on its growth).  
\end{rmq}

\begin{rmq}
    In view of rmq~\ref{rem:mass-bound-below}, if~\eqref{eq:liminf-prob-gel} is satisfied with the indicators $\mathbf{1}_{m \geq 1}$ replaced by any $\mathbf{1}_{m \geq c}$ for some $c > 0$, and~\eqref{eq:class-gel-sum-cond} is satisfied with $\cstone(j):= \inf_{x,y \in [c2^{j}, c2^{j+1})}\bar{K}(x,y)$, then we can also show that stochastic gelation occurs. 
\end{rmq}
\begin{rmq} \label{rem_cor} We can more precisely quantify the function $\psi$ involved in the $(\psi,\delta)$-stochastic gelation time (i.e., the function describing the sizes of the large particles contributing to gelation - see Definition~\ref{defn:stoch-gelation}). In particular we can show that $T^{\psi,\delta}_{g}<\infty$ for any $\psi(N)\leq N^{b}$, and  any $b\in(0,1)$. See also Example~\ref{ex_final}. \end{rmq}
\subsection{Gelation criteria via coupling with inhomogeneous random graphs} \label{sec:coupling}
Here we fix $K_{N} \equiv K$ for some kernel $K$.
It is well-known that for the classical multiplicative kernel $\bar K(x,y) = m(x) m(y)$, if the process starts with $N$ particles of mass $1$ at time $t=0$, then  the cluster masses at time $t > 0$ are in one-to-one correspondence with the sizes of components of a continuous time analogue of the Erd\H{o}s--R\'enyi random graph, where edges appear between pairs of vertices at rate $1/N$. Likewise, the work by Patterson and Heydecker~\cite{heydecker2019bilinear} shows that for cluster coagulation models of any `bilinear' type (see Section~\ref{example_CGP}), the cluster masses at any time $t$ are in one-to-one correspondence with the component sizes of an \emph{inhomogenous random graph}. Thus, in regimes where a `giant component' arises in the inhomogeneous random graph model, gelation occurs in the coagulation process, with an explicit description of the gelation time.

Theorem~\ref{gel_uppb} below generalises these results. 
First, we define some terminology required to make a connection with inhomogeneous random graphs (using the established results from~\cite{inhom-rand-graph-giant}). We call a coagulation kernel $K$ \emph{graph dominating} if
\[
\text{for all $x,y,q \in E$, $K(x,y,\cdot)$-a.a. $z$} \quad \text{ we have } \quad \bar{K}\left(z, q\right) \geq  \bar{K}(x,q) + \bar{K}(y,q).
\]
If the opposite inequality holds, we say the kernel is \emph{graph dominated}. 

In Theorem~\ref{gel_uppb} we assume the following. First, the initial conditions are \emph{monodispersed}, i.e., $\supp{\big(\mathbf{L}^{(N)}_{0}\big)} \subseteq \{x \in E: m(x) = 1\}$ for all $N\in\mathbb{N}$. Next, assume that there exists a deterministic $\bar{\mathbf{L}}^{*}_{0}\in \mathcal{M}_{+}(E)$ such that 
\begin{equation}\label{conv_IC}
\bar{\mathbf{L}}^{(N)}_{0} \rightarrow \bar{\mathbf{L}}^{*}_{0} \quad \text{weakly, in probability.}
\end{equation}
Moreover, assume that the mass function $m$ is continuous and $E$ is a separable metric space, on which $\bar{\mathbf{L}}^{*}_{0}$ is a Borel probability measure.  
Note that when the initial conditions are monodispersed and~\eqref{conv_IC} holds then $\supp{\big(\bar{\mathbf{L}}^{*}_{0}\big)} \subseteq \{x \in E: m(x) = 1\}$. 
Now, define the following quantities 
\begin{linenomath}
\begin{align} \label{eq:graph-operator-def}
 T_{\bar{K}, \bar{\mathbf{L}}_{0}^{*}}f(x) := \int_{E} f(y)\bar{K}(x, y)\bar{\mathbf{L}}_{0}^{*}(\dd y), \quad \Sigma(\bar{K}, \bar{\mathbf{L}}_{0}^{*}) &:= \sup_{f \in L^{2}(\bar{\mathbf{L}}_{0}^{*}), \|f\|_{L^{2}(\bar{\mathbf{L}}_{0}^{*})} = 1} \left\|T_{\bar{K}, \bar{\mathbf{L}}_{0}^{*}}f\right\|_{L^{2}(\bar{\mathbf{L}}_{0}^{*})}
\\ \text{ and} \quad t^{*} &  := \inf\left\{t > 0: t \Sigma(K, \bar{\mathbf{L}}_{0}^{*}) > 1\right\}.
\end{align}
\end{linenomath}
\begin{thm}\label{gel_uppb}
    Suppose that the above conditions are satisfied. Then the following hold.
    \begin{enumerate}
        \item \label{item:t_g} If $(\bar{\mathbf{L}}^{(N)}_t)_{t\geq 0}$ is graph dominating, then strong gelation occurs, with strong gelation time at most $t^*$. 
        \item If $(\bar{\mathbf{L}}^{(N)}_t)_{t\geq 0}$ is graph dominated then, for any $\psi: \mathbb{N} \rightarrow \mathbb{N}$ non-decreasing with $\lim_{N \to \infty} \psi(N) = \infty$, and for any $\delta > 0$ we have $T^{\psi, \delta}_{g} \geq t^{*}$.
    \end{enumerate}
\end{thm}
\hfill $\blacksquare$

\begin{rmq}
Theorem~\ref{gel_uppb} works by coupling the coagulation process with the following inhomogeneous random graph process. The vertices of the graph are given by the initial clusters, with atoms of $\mathbf{L}^{(N)}_{0}$ determining \emph{types} of the initial set of vertices (all of mass $1$). Then, edges in the graph appear between vertices of type $x, y \in E$ at rate $\bar{K}(x,y)$. The condition of being graph dominating (similarly graph dominated) allows this coupling to be carried out in such a way that a coagulation between clusters precedes the appearance of an edge between associated connected components in the inhomogeneous random graph model. Thus, gelation occurs before the appearance of a giant component in the graph, well-known to emerge at time $t^{*}$. As this coupling argument is quite straightforward, we omit the details in this paper (see, however, version 2 of the extended pre-print~\cite{AnIyMa23gel}). Nevertheless, we believe that examples, such as the following, mean that this monotonicity is important to be aware of in applied contexts.  
\end{rmq}

\begin{example} \label{lem:coupl-admiss} Suppose we are in the same setting as Example~\ref{ex:toy-spatial-coag-1} with clusters moving to their centre of mass. Let $\rho: \mathcal{S} \rightarrow [0, \infty)$ be an even, concave function. Then the coagulation process with kernel $\bar{K}:(\mathcal{S}\times \mathbb{N}_0)^2 \rightarrow [0, \infty)$ defined by $\bar{K}((p, m), (s, n)) := m n\rho(p-s)$
is graph dominating.
\end{example}
\mayberemove{ \begin{example}
Consider the toy spatial coagulation models, introduced above, where $E=\mathcal{S}\times (0, \infty)$. Suppose that we begin with $N$ clusters (of mass $1$, say), sampled i.i.d from the uniform distribution on the hypercube $\mathcal{S}= [0,1]^{d}$, and consider the kernel, 
\[
\bar{K}((p,n), (s, o)) := \begin{cases}
\frac{\kappa_{0}}{(||p-s||)^{\alpha}}, & \text{if } p \neq s \\
0, & \text{otherwise;}
\end{cases}
\] 
with $\kappa_0\in (0,\infty)$ a constant and $\alpha >0$.
Note that if clusters merge at a constant rate $\kappa_0$ (without influence of distance), it is well-known that gelation does not occur, however, in this model, the distance now plays an important role, and one may readily verify that if $\alpha/ d > 1$, stochastic gelation occurs.  
Indeed, one may readily  verify (for example, by using induction), that almost surely, the initial configuration of clusters is such that the distance between any two points is positive throughout the dynamics of the coagulation process, and hence~\eqref{eq:mart-technical-bound} is satisfied almost surely. 
Now, for each $j\in\{0,1,2,\dots,\log_{2}(\psi(N))\}$ we take a partition of $\mathcal{S}$ that consists of $\xi(N)$ hypercubes with side-length $\frac{1}{(\xi(N))^{1/d}}$. Note that $c'(P,j)= \kappa_0\min_{p,s\in P}\frac{1}{(||p-s||)^{\alpha}}\geq \xi(N)^{\alpha/d}$, for all $P$. Fix, for example, $f_{j}:=2^{-j}$; then the sum in \eqref{eq:summability} reads 
\[
\limsup_{N \to \infty} \frac{\xi(N)}{(\xi(N))^{\alpha/d}}\sum_{j=0}^{\log_{2}\psi(N)} 2^{3j} \leq C \limsup_{N \to \infty} {(\xi(N))^{1-\alpha/d}} (\psi(N))^3 = \limsup_{N \to \infty}C\frac{\psi(N)^{3}}{\xi(N)^{\alpha/d -1}},
\]
for some $C\in\mathbb{R}^{+}$. Now if $\frac{\alpha}{d}> 1$ and we set $\psi(N):=\xi(N)^{\frac{\alpha/d -1}{3}}$; we can choose any $\xi(N)$ such that $\lim_{N \to \infty} \xi(N) = \infty$, and that fulfils the third condition in Assumption \ref{ass:gelation}. 
\end{example}
\begin{example}\label{Ex:mult_plus_dist}
Consider again the toy spatial coagulation models, introduced above, where $E=\mathcal{S}\times (0, \infty)$. A natural choice of kernel may be to choose a function that is a product of a non-increasing function of the distance $d$ between clusters (clusters interact more quickly if they are closer together), and a function of their mass. 
In this manner, suppose we choose kernel of product form   $\bar{K}((p,n), (s, o)) := h(d(p,s)) W(n,o)$ where $h: [0, \infty) \rightarrow [0, \infty)$ is non-increasing and non-zero, and $W$ is continuous and  satisfies the conditions under which Corollary \ref{cor:classical-gelation} applies. As $h$ is bounded from below (by $c_{0}$ say) on an interval $[0, \eps)$, for $\eps$ sufficiently small, and $\mathcal{S}$ is compact, we can choose a finite partition $\mathscr{P}$ consisting of open balls of radius $\eps$, such that, if $P \in \mathscr{P}$ for all $x, y \in P$ we have $\bar{K}((p,n), (s, o)) \geq  c_{0} W(n,o)$. By choosing $\mathscr{P}^{(j)} = \mathscr{P}$ for each $j$, one readily verifies (in a similar manner to the proof of Corollary~\ref{cor:classical-gelation}) that all conditions in Assumption \ref{ass:gelation} are verified. Finally, if we begin, for example, with $N||\pi^{(N)}||=N$ clusters of mass $1$, the maximum value of $m(x)$ equals $N$ for all $x\in E$, and we can easily verify condition \eqref{eq:mart-technical-bound} for $\bar{K}$, by setting $g_{\pi}=N \max_{(x,y) \in (\mathcal{S} \times [0,N])^2} \bar{K}(x,y) $.
\end{example}}

\section{Proofs of main results} \label{sec:main-results-proofs}
This section is dedicated to the proofs of Theorem~\ref{thm:simp-gel} and Corollary~\ref{cor:classical-gelation}. 
For the proof of Item~\ref{item:thme-simp-gel-1} of Theorem~\ref{thm:simp-gel} we fix a sequence of initial conditions $\{\boldsymbol{\pi}^{(N)}\}_{N\in\mathbb{N}}$ satisfying the assumptions of the theorem for some $\eps, \rho_0 > 0$, and fix $\delta \in (0, \rho_0\eps)$ (this is the threshold involved in definition~\eqref{eq:giant-time-coagulation}). Recall from~\eqref{eq:prob_pi_N} that $\mathbb{P}_{N, \boldsymbol{\pi}^{(N)}}(\cdot)$ and $\mathbb{E}_{N, \boldsymbol{\pi}^{(N)}}[\cdot]$ denote probabilities and expectations associated with a coagulation process that starts with such an initial condition. 

Since $\delta \in (0, \eps \rho_{0})$ we have $\delta = \eps \rho$, for some $\rho \in (0, \rho_{0})$.
Note that if $(f_{k})_{k \in \mathbb{N}_{0}}$ satisfies the requirement of Item~\ref{item:condition_diagonal} of Assumption~\ref{ass:gelation}, we may re-scale $(f_{k})_{k \in \mathbb{N}_{0}}$ so that $\sum_{i=0}^{\infty} f_{i} = \rho_{0} - \rho>0$. This re-scaling does not affect Assumption~\ref{ass:gelation}. Therefore, with $(f_{i})_{i \in \mathbb{N}_{0}}$ defined in this way, define a strictly decreasing sequence $(\rho_{k})_{k\in\mathbb{N}_{0}}$ as follows: define $\rho_0$ according to~\eqref{eq:initial-condition}, and for $k \in \mathbb{N}$ set
\begin{equation}\label{eq:rho_k}
\rho_k :=\rho+\sum_{i\geq k} f_i,
\end{equation}
so that $\rho_{k} \downarrow \rho>0$.

We recall that in Theorem~\ref{thm:simp-gel} we assume that there exists a function $\psi(N)$ satisfying Assumption~\ref{ass:gelation}. Ideally, with $\psi'$ being the function appearing in Item~\ref{item:thme-simp-gel-1} of Theorem~\ref{thm:simp-gel}, we would have $\psi' = \psi$. However, from a technicality in the proof of Item~\ref{item:thme-simp-gel-1} of Theorem~\ref{thm:simp-gel} (more specifically to apply Lemma~\ref{lem:easy-2}) we need to define $\psi'$ by 
\begin{equation}\label{kappa}
\psi'(N) := \psi(N) \wedge \kappa(N) \quad \text{where} \quad \kappa(N):=\max\left\{k\in\mathbb{N}: \frac{2 k}{ \min_{0\leq h\leq \log_{2}(k)-1}f_h}\leq \frac{N \mass{\boldsymbol{\pi}^{(N)}}}{\xi(N)}\right\}.
\end{equation}
In this way we ensure that the following inequality is satisfied:
\begin{equation}\label{eq:ass_altern1}
\frac{2 \psi'(N)}{ \min_{0\leq h\leq \log_{2}(\psi'(N))-1}f_h}\leq \frac{N \mass{\boldsymbol{\pi}^{(N)}}}{\xi(N)} \quad \text{for all $N \in \mathbb{N}$.}\end{equation}
The function $\psi'(N)$ satisfies the conditions of Assumption~\ref{ass:gelation} since we have $\lim_{N \to \infty} \psi'(N) = \infty$. Indeed, by definition, 
\begin{equation} \label{eq:kappa-ineq}
\liminf_{N \to \infty} \frac{\kappa(N) + 1}{\min_{0\leq h\leq \log_{2}(\kappa(N) + 1) -1}f_h} >  \liminf_{N \to \infty} \frac{N \mass{\boldsymbol{\pi}^{(N)}}}{2\xi(N)} = \infty,
\end{equation}
where the last equality follows by Item~ \ref{item:condition_pidgeon} of Assumption~\ref{ass:gelation}.
If $\kappa(N)$ were bounded infinitely often, the left-side of~\eqref{eq:kappa-ineq} would also be bounded. Hence $\lim_{N \to \infty} \kappa(N) = \infty$, and since $\lim_{N \to \infty} \psi(N) = \infty$, from \eqref{kappa} we also have $\lim_{N \to \infty} \psi'(N) = \infty$.

Now, in order to prove Theorem~\ref{thm:simp-gel}, we define the family of functions $F_{k}: 
\ME \rightarrow \mathbb{R}^{+}$ such that, for each $\mathbf{L}\in \ME$
\begin{equation} \label{eq:eff-k}
F_{k}(\mathbf{L}) :=\frac{\left\langle m \mathbf{1}_{m \geq 2^{k+1}}, \mathbf{L} \right\rangle}{\mass{\boldsymbol{\pi}^{(N)}}}. 
\end{equation}
We also define an associated family of stopping times $(\mathcal{T}_k)_{k \in \mathbb{N}}$ such that
\begin{equation}\label{stopT}
\mathcal{T}_{k} := \inf\left\{t > 0: \left\langle m \mathbf{1}_{m \geq 2^{i}}, \bar{\mathbf{L}}^{(N)}_{t} \right\rangle /\mass{\boldsymbol{\pi}^{(N)}} \geq \rho_{i} \text{ for all } i= 0, 1, \ldots, k\right\}.
\end{equation}
In words, $\mathcal{T}_{k}$ represents the first time that the total mass of clusters with mass at least $2^{i}$ exceeds $\rho_{i}\mass{\boldsymbol{\pi}^{(N)}}$ for each $i = 0, \ldots, k$. 
Note that the functions $F_k$ and the times $\mathcal{T}_{k} $ depend on $\boldsymbol{\pi}^{(N)}$ and   on $N$, but for brevity of notation, we will exclude this dependence in the remainder of the section. 

\begin{lemma} \label{lem:easy-2}
Suppose that $\vec{v} = (v_1, \ldots, v_n)$, $\vec{c} = (c_1, \ldots, c_n)$ are such that $v_h \in \mathbb{N}$ and $c_{h} > 0$, for all $h=1,\dots,n$. Then, if  $\sum_{h=1}^{n} v_{h} \geq \kappa_1 > n$ and $\sum_{h=1}^{n} \frac{1}{c_h} = \kappa_2$ then
\begin{equation} \label{eq:identity-easy2}
	\sum_{h} c_k \left(v_h^2 - v_h \right) \geq \frac{(\kappa_1- n)^2}{2\kappa_2}. 
\end{equation}
\end{lemma}
\begin{proof}
Re-writing the left side of Equation~\eqref{eq:identity-easy2} we get
\begin{equation} \label{eq:opt-prob}
\begin{aligned}
	\sum_{h} c_{h} (v_h^2 - v_h)\mathbf{1}_{\{v_{h} > 1\}} \geq \frac{1}{2}\sum_{h}  \frac{(v_{h} \mathbf{1}_{\{v_{h} \geq 2\}})^2}{\sfrac{1}{c_{h}}}\geq \frac{\left(\sum_{h}v_{h}\mathbf{1}_{\{v_{h} \geq 2\}}\right)^2}{2\kappa_2},
\end{aligned}
\end{equation}
where the last inequality is due to the Cauchy--Schwarz inequality. 
Since $\sum_{h} v_{h} \mathbf{1}_{\{v_{h} \geq 2\}} \geq \kappa_1 -n\geq 1$, we deduce the result.

\end{proof}
We now prove Theorem~\ref{thm:simp-gel}.
\subsection{Proof of Item~\ref{item:thme-simp-gel-1} of Theorem~\ref{thm:simp-gel}}
\begin{proof}
We first fix a positive integer $k \leq  \left \lceil \log_2(\psi'(N)) \right \rceil - 1$, where we recall that the function $\psi'$ comes from~\eqref{kappa} (the reason for this choice of $k$ is to be able to apply Item~\ref{item:partition_lower_bound} of Assumption~\ref{ass:gelation} later on).
Note that the function $F_{k}$ defined in \eqref{eq:eff-k} is bounded above by $1$ almost surely on the path of the process, and is clearly measurable. Now, recalling the generator in~\eqref{eq:gen-def-2},  as $F_{k}$ only considers the masses of clusters larger than $2^{k+1}$, ignoring other features, and the mass of the coagulated cluster $z$ is $m(z) = m(x) + m(y)$, we deduce that 
\begin{linenomath}
\begin{align} \label{eq:main-display-1}
 &\tilde{\mathcal{A}}_NF_{k}(\bar{\mathbf{L}}^{(N)}_{t}) =  \frac{1}{2\mass{\boldsymbol{\pi}^{(N)}}} \int_{E\times E} \bar{\mathbf{L}}^{(N)}_{t}(\dd x) \left(\bar{\mathbf{L}}^{(N)}_{t} - \frac 1N\delta_{x}\right)(\dd y) \bar{K}_{N}(x,y)  \\  & \hspace{1cm}\times \left[(m(x)+m(y))\mathbf{1}(m(x) + m(y) \geq 2^{k+1}) - m(x)\mathbf{1}(m(x) \geq 2^{k+1}) - m(y) \mathbf{1}(m(y) \geq 2^{k+1})\right]. 
\end{align}
\end{linenomath}
Now, 
\begin{equation} \label{eq:mart-defn}
M^{(N)}_{F_{k}}(t) := F_{k}(\bar{\mathbf{L}}^{(N)}_{t}) - F_{k}(\boldsymbol{\pi}^{(N)}) - \int_{0}^{t} \tilde{\mathcal{A}}_{N}F_{k}(\bar{\mathbf{L}}^{(N)}_{s}) \dd s
\end{equation}
is a martingale with respect to the natural filtration of the process\footnote{The fact that $M^{(N)}_{F_{k}}(t)$ is a martingale follows from the definition of the infinitesimal generator of the process, see, for example,~\cite[Proposition~7.1.6]{revuzyor}} initiated by the measure $\boldsymbol{\pi}^{(N)}$. Let $I, J > 0$ be given. By Doob's optional sampling theorem (see, for example,~\cite[Corollary~2.3.6 and Theorem~2.3.2]{revuzyor}) applied to the bounded stopping times $I \wedge \mathcal{T}_{k} \leq I \wedge \mathcal{T}_{k+1}$ we have 
\begin{linenomath}
\begin{align}
\E[N, \boldsymbol{\pi}^{(N)}]{F_{k}(\bar{\mathbf{L}}_{I \wedge \mathcal{T}_{k+1}})} & = \E[N, \boldsymbol{\pi}^{(N)}]{F_{k}(\bar{\mathbf{L}}_{I \wedge \mathcal{T}_{k}})} + \E[N, \boldsymbol{\pi}^{(N)}]{\int_{I \wedge \mathcal{T}_{k}}^{I \wedge \mathcal{T}_{k+1}} \tilde{\mathcal{A}}_N F_{k}(\bar{\mathbf{L}}_t) \dd t} \\ & \geq \E[N, \boldsymbol{\pi}^{(N)}]{\int_{I \wedge \mathcal{T}_{k}}^{I \wedge \mathcal{T}_{k+1}} \tilde{\mathcal{A}}_N F_{k}(\bar{\mathbf{L}}_t) \dd t} \geq \E[N, \boldsymbol{\pi}^{(N)}]{\mathbf{1}_{\mathcal{T}_{k} \leq J} \int_{I \wedge \mathcal{T}_{k}}^{I \wedge \mathcal{T}_{k+1}} \tilde{\mathcal{A}}_N F_{k}(\bar{\mathbf{L}}_t) \dd t}. 
\end{align}
\end{linenomath}
Noting that $\tilde{\mathcal{A}}_NF_{k}$ is non-negative and the integrand on the left is monotone in $I$, by monotone convergence in $I$, and Fatou's lemma we have 
\begin{linenomath}
    \begin{align}
        \E[N, \boldsymbol{\pi}^{(N)}]{F_{k}(\bar{\mathbf{L}}_{\mathcal{T}_{k+1}})} & \geq \E[N, \boldsymbol{\pi}^{(N)}]{\mathbf{1}_{\mathcal{T}_{k} \leq J} \liminf_{I \to \infty} \int_{I \wedge \mathcal{T}_{k}}^{I \wedge \mathcal{T}_{k+1}} \tilde{\mathcal{A}}_N F_{k}(\bar{\mathbf{L}}_t) \dd t} \\ & = \E[N, \boldsymbol{\pi}^{(N)}]{\mathbf{1}_{\mathcal{T}_{k} \leq J} \int_{\mathcal{T}_{k}}^{\mathcal{T}_{k+1}} \tilde{\mathcal{A}}_N F_{k}(\bar{\mathbf{L}}_t) \dd t}.
    \end{align}
\end{linenomath}
Now, by applying monotone convergence in $J$, we deduce that 
\begin{equation} \label{Dynkins}
    \E[N, \boldsymbol{\pi}^{(N)}]{F_{k}(\bar{\mathbf{L}}_{\mathcal{T}_{k+1}})} \geq \E[N, \boldsymbol{\pi}^{(N)}]{\mathbf{1}_{\mathcal{T}_{k} < \infty} \int_{\mathcal{T}_{k}}^{\mathcal{T}_{k+1}} \tilde{\mathcal{A}}_N F_{k}(\bar{\mathbf{L}}_t) \dd t}.
\end{equation}

We now seek lower bounds for the quantity $\tilde{\mathcal{A}}_N F_{k}(\bar{\mathbf{L}}_t)$ appearing in~\eqref{Dynkins}, when $\mathcal{T}_{k} \leq t < \mathcal{T}_{k+1}$ (note that $\mathcal{T}_{k+1}$ might be infinite). First, note that 
\begin{equation} \label{eq:bound-negative-term}
m(x)\mathbf{1}(m(x) \geq 2^{k+1}) + m(y) \mathbf{1}(m(y) \geq 2^{k+1}) \leq (m(x)+m(y))\mathbf{1}(m(x) \text{ or } m(y) \geq 2^{k+1}).
\end{equation} 
Using~\eqref{eq:bound-negative-term}, we may bound~\eqref{eq:main-display-1} from below so that
 \begin{linenomath}
\begin{align} \label{eq:int-1}
 &  \tilde{\mathcal{A}}_NF_{k}(\bar{\mathbf{L}}^{(N)}_t) \geq 
 \frac{1}{2\mass{\boldsymbol{\pi}^{(N)}}}\int_{E\times E} \bar{\mathbf{L}}^{(N)}_t(\dd x ) \left(\bar{\mathbf{L}}^{(N)}_t - \frac 1N\delta_{x}\right)(\dd y)\bar{K}_{N}(x,y) 
\\ \nonumber & \hspace{6cm} \times
(m(x)+m(y)) \mathbf{1}(m(x)+m(y) \geq 2^{k+1} > m(x),m(y))
\\ &  \geq \frac{2^{k}}{\mass{\boldsymbol{\pi}^{(N)}}}\int_{E\times E} \bar{\mathbf{L}}^{(N)}_t(\dd x ) \left(\bar{\mathbf{L}}^{(N)}_t - \frac 1N\delta_{x}\right)(\dd y)\bar{K}_{N}(x,y) \mathbf{1}(m(x)+m(y) \geq 2^{k+1} > m(x),m(y))
\\ & \geq \frac{2^{k}}{\mass{\boldsymbol{\pi}^{(N)}}} \times \\ & \hspace{1cm}\sum_{P \in \mathscr{P}_{N}^{(k)}}  \int_{P\times P} \bar{\mathbf{L}}^{(N)}_t(\dd x ) \left(\bar{\mathbf{L}}^{(N)}_t - \frac 1N\delta_{x}\right)(\dd y)\bar{K}_{N}(x,y) \mathbf{1}(m(x)+m(y) \geq 2^{k+1} > m(x),m(y)), 
\end{align}
\end{linenomath}
where the last inequality follows from restricting the integral from the space $E\times E$ to the space $\bigcup_{P \in \mathscr{P}_{N}^{(k)}}P\times P$, with the partition $\mathscr{P}_{N}^{(k)}$ as defined in Item~\ref{item:partition_lower_bound} of Assumption~\ref{ass:gelation}. Now, using the definition of  $\cstone_{N}(P,k)$ in Item~\ref{item:partition_lower_bound} of Assumption~\ref{ass:gelation} we bound this further by 
\begin{align*} \numberthis \label{eq:part-lowered}
& \frac{2^{k}}{\mass{\boldsymbol{\pi}^{(N)}}} \times\\ & \hspace{1cm}\sum_{P \in \mathscr{P}_{N}^{(k)}} \cstone_{N}(P,k) \int_{P\times P} \bar{\mathbf{L}}^{(N)}_t(\dd x ) \left(\bar{\mathbf{L}}^{(N)}_t - \frac 1N\delta_{x}\right)(\dd y) \mathbf{1}(m(x)+m(y) \geq 2^{k+1} > m(x),m(y)).
\end{align*}
Moreover, since $P \in \mathscr{P}_{N}^{(k)}$, by Item~\ref{item:partition_lower_bound} of Assumption~\ref{ass:gelation} for $(x, y) \in P \times P$ we have $2^{k} \leq m(x), m(y) < 2^{k+1}$. This implies that the indicator $\mathbf{1}(m(x)+m(y) \geq 2^{k+1} > m(x),m(y))$ in~\eqref{eq:part-lowered} is one. Thus, we can re-write~\eqref{eq:part-lowered} as 
\begin{linenomath}
    \begin{align}
        & \frac{2^{k}}{\mass{\boldsymbol{\pi}^{(N)}}} \sum_{P \in \mathscr{P}_{N}^{(k)}} \cstone_{N}(P,k)\int_{P} \bar{\mathbf{L}}^{(N)}_t(\dd x ) \left(\left( \int_{P} \bar{\mathbf{L}}^{(N)}_t(\dd x )\right) - \frac{1}{N}\right).
    \end{align}
\end{linenomath}

For brevity of notation, we now write $\left\langle f, \bar{\mathbf{L}}^{(N)}_{t} \right\rangle$ for terms of the form \[\int_{E} \bar{\mathbf{L}}^{(N)}_t(\dd x) f(x) \quad \text{ or } \quad \int_{E} \bar{\mathbf{L}}^{(N)}_t(\dd y) f(y).\]  Then, summing over the possible values $m(x), m(y) \in [2^{k}, 2^{k+1})$, we deduce the lower bound
 \begin{linenomath} 
\begin{align*} 
 & \tilde{\mathcal{A}}_NF_{k}(\bar{\mathbf{L}}^{(N)}_t) \geq \frac{2^{k}}{\mass{\boldsymbol{\pi}^{(N)}}} \sum_{P \in \mathscr{P}_{N}^{(k)}} \sum_{n_{1},n_{2} = 2^{k}}^{2^{k+1} -1} \cstone_{N}(P,k)
\left\langle \mathbf{1}_{n_{1} \leq m(x) < n_{1}+1, x\in P}, \bar{\mathbf{L}}^{(N)}_{t} \right\rangle 
\\  &\hspace{8cm} \times
\left(\left\langle \mathbf{1}_{n_{2} \leq m(x) < n_{2}+1, x\in P}, \bar{\mathbf{L}}^{(N)}_{t} \right\rangle - \frac 1N\mathbf{1}_{\{n_{1}=n_{2}\}}\right)
\\ \numberthis \label{eq:disp-2} & =\frac{2^{k}}{\mass{\boldsymbol{\pi}^{(N)}}} 
\\& \times \frac 1{N^2} \sum_{P \in \mathscr{P}_{N}^{(k)}} \cstone_N(P,k)\Bigg(\Big(N\sum_{n=2^{k}}^{2^{k+1}-1} \left\langle \mathbf{1}_{n \leq m(x) < n+1, x\in P}, \bar{\mathbf{L}}^{(N)}_{t} \right\rangle\Big)^2 - N\sum_{n=2^{k}}^{2^{k+1}-1}\left\langle \mathbf{1}_{n \leq m(x) < n+1, x\in P}, \bar{\mathbf{L}}^{(N)}_{t} \right\rangle \Bigg).
\end{align*}
\end{linenomath} 
We now wish to apply Lemma~\ref{lem:easy-2} to the sum indexed by the elements $P$ of $\mathscr{P}_{N}^{(k)}$ in the above display, where the integer valued random variables $\sum_{n=2^{k}}^{2^{k+1}-1} N\left\langle \mathbf{1}_{n \leq m(x) < n+1, x\in P}, \bar{\mathbf{L}}^{(N)}_{t} \right\rangle$ play the role of $v_{h}$ in the lemma. 
These integers count the number of clusters at time $t$ in each element $P$ of the partition $\mathscr{P}^{(k)}_N$.
Since $|\mathscr{P}_{N}^{(k)}|\leq \xi(N)$, $\xi(N)$ plays the role of $n$ appearing in Lemma~\ref{lem:easy-2}.

Note, that, if $\mathcal{T}_{k} \leq t < \mathcal{T}_{k+1}$, by the definition in~\eqref{stopT}, we have 
\begin{equation}\label{ineq_F}
\left\langle m \mathbf{1}_{m \geq 2^{k}}, \bar{\mathbf{L}}^{(N)}_{t} \right\rangle \geq \rho_{k} \mass{\boldsymbol{\pi}^{(N)}} \quad \text{ and } 
\left\langle m \mathbf{1}_{m \geq 2^{k+1}}, \bar{\mathbf{L}}^{(N)}_{t} \right\rangle < \rho_{k+1} \mass{\boldsymbol{\pi}^{(N)}}, 
\end{equation}
so that
\begin{equation} \label{eq:between-stops-1}
 \left\langle m \mathbf{1}_{2^{k} \leq m < 2^{k+1}}, \bar{\mathbf{L}}^{(N)}_{t} \right\rangle
 \stackrel{\eqref{eq:rho_k}}{\geq} f_{k} \mass{\boldsymbol{\pi}^{(N)}}.  
 \end{equation}
Therefore, we choose $\kappa_{1}$ in Lemma~\ref{lem:easy-2} such that 
\begin{align*}
\sum_{h=1}^{n} v_{h} = \sum_{P\in\mathscr{P}_{N}^{(k)}} \sum_{n=2^{k}}^{2^{k+1}-1} N\left\langle \mathbf{1}_{n \leq m(x) < n+1, x\in P}, \bar{\mathbf{L}}^{(N)}_{t} \right\rangle 
& =  N \left\langle \mathbf{1}_{2^{k} \leq m < 2^{k+1}}, \bar{\mathbf{L}}^{(N)}_{t} \right\rangle \\ & \stackrel{\eqref{eq:between-stops-1}}{\geq} N \frac{f_{k} \mass{\boldsymbol{\pi}^{(N)}}}{2^{k+1}}  =: \kappa_1
\end{align*}

For the assumptions of Lemma~\ref{lem:easy-2} to hold, we  further require
\begin{equation}\label{appl_lemma}
\kappa_{1} = N \frac{f_{k} \mass{\boldsymbol{\pi}^{(N)}}}{2^{k+1}}
>\xi(N).
\end{equation}
This follows from the fact that 
\begin{linenomath}
 \begin{align}
\kappa_1 =  N \frac{f_{k} \mass{\boldsymbol{\pi}^{(N)}}}{2^{k+1}} & \geq \frac{(\min_{0\leq \ell \leq \left \lceil \log_{2}(\psi'(N)) \right \rceil - 1}f_{\ell}) N\mass{\boldsymbol{\pi}^{(N)}}}{\psi'(N)} \\ & \stackrel{\eqref{eq:ass_altern1}}{\geq} 2\xi(N) > \xi(N)
\label{ineq_lemma2}
 \end{align}
 \end{linenomath}
 where the first inequality in \eqref{ineq_lemma2} is obtained by inserting the value $k = \log_2(\psi'(N)) - 1$ in the denominator.
Therefore, applying Lemma~\ref{lem:easy-2} to the sum in  \eqref{eq:disp-2}, with $\kappa_{2} := \sum_{P \in \mathscr{P}_{N}^{(k)}} \cstone_N(P,k)^{-1}$ we obtain the lower bound:
\begin{align}\label{eq:generator_bound} & \tilde{\mathcal{A}}_NF_{k}(\bar{\mathbf{L}}^{(N)}_t) \\ & \hspace{0.5cm} \geq \frac{2^{k}}{\mass{\boldsymbol{\pi}^{(N)}} N^2} \times \frac{(\kappa_1 - n)^2}{2\kappa_2} = \frac{2^{k-1}\left(\mass{\boldsymbol{\pi}^{(N)}}\frac{f_{k}}{2^{k+1}} -
\frac{\xi(N)}N \right)^2 }{\mass{\boldsymbol{\pi}^{(N)}} \sum_{P \in \mathscr{P}_{N}^{(k)}} \cstone_N(P,k)^{-1}}, \quad \forall \,t\in [\mathcal{T}_{k}, \mathcal{T}_{k+1}). 
\end{align}
Since ~\eqref{ineq_lemma2} implies that $\frac{\xi(N)}N \leq  \frac 12 \mass{\boldsymbol{\pi}^{(N)}}\frac{f_{k}}{2^{k+1}}$, 
we finally deduce the bound
\[
\tilde{\mathcal{A}}_NF_{k}(\bar{\mathbf{L}}^{(N)}_t) \geq \frac{f_{k}^2 \mass{\boldsymbol{\pi}^{(N)}}}{2^{k+5}(\sum_{P \in \mathscr{P}_{N}^{(k)}} \cstone_N(P,k)^{-1})}, \quad \text{for all $k\leq \log_2(\psi'(N)) -1$ and $t\in [\mathcal{T}_{k}, \mathcal{T}_{k+1})$.}
\] 
Now, recalling that $0\leq F_{k}(\bar{\mathbf{L}}^{(N)}_t)\leq 1$,
by applying~\eqref{Dynkins} 
we have 
\begin{linenomath}
\begin{align*}
1 & \geq \E[N, \boldsymbol{\pi}^{(N)}]{F_{k}(\bar{\mathbf{L}}^{(N)}_{\mathcal{T}_{k+1}})} \\ & \geq \E[N, \boldsymbol{\pi}^{(N)}]{\mathbf{1}_{\mathcal{T}_{k} < \infty} \int_{\mathcal{T}_{k}}^{\mathcal{T}_{k+1}}\tilde{\mathcal{A}}F_{k}(\bar{\mathbf{L}}^{(N)}_t) \dd t}  \geq \frac{f_{k}^2\mass{\boldsymbol{\pi}^{(N)}}}{2^{k+5}(\sum_{P \in \mathscr{P}_{N}^{(k)}} \cstone_N(P,k)^{-1})}\E[N, \boldsymbol{\pi}^{(N)}] {\mathbf{1}_{\mathcal{T}_{k} < \infty}(\mathcal{T}_{k+1}-\mathcal{T}_{k})}
\end{align*}
\end{linenomath}
which in turn yields
\begin{equation} \label{eq:lacunary-bounded-initial-stop}
\E[N, \boldsymbol{\pi}^{(N)}] {\mathbf{1}_{\mathcal{T}_{k} < \infty}(\mathcal{T}_{k+1}-\mathcal{T}_{k})}\leq \frac{2^{k+5}(\sum_{P \in \mathscr{P}_{N}^{(k)}} \cstone_N(P,k)^{-1})}{\mass{\boldsymbol{\pi}^{(N)}} f_{k}^2}. 
\end{equation}
Now, by assumption in Equation~\eqref{eq:initial-condition}, we have $\mathcal{T}_{0} = 0 < \infty$ almost surely. Thus, by induction and the bound in~\eqref{eq:lacunary-bounded-initial-stop}, for all $k \leq \lceil\log_2(\psi'(N))\rceil - 1$, we have $\E[N, \boldsymbol{\pi}^{(N)}]{\mathcal{T}_{k}}<\infty$ and hence
$\mathcal{T}_{k} < \infty$ almost surely. Therefore, for all $k \leq \lceil\log_2(\psi'(N))\rceil - 1$ we may re-write~\eqref{eq:lacunary-bounded-initial-stop} as 
\begin{equation} \label{eq:lacunary}
\E[N, \boldsymbol{\pi}^{(N)}] {\mathcal{T}_{k+1}-\mathcal{T}_{k}}\leq \frac{2^{k+5}(\sum_{P \in \mathscr{P}_{N}^{(k)}} \cstone_N(P,k)^{-1})}{\mass{\boldsymbol{\pi}^{(N)}} f_{k}^2}. 
\end{equation}
Then, by summing~\eqref{eq:lacunary}, we deduce that
\begin{linenomath}
    \begin{align*} 
        \E[N, \boldsymbol{\pi}^{(N)}]{\mathcal{T}_{\lceil\log_{2}(\psi'(N))\rceil}} &=  \sum_{k=0}^{\lceil\log_2(\psi'(N))\rceil-1}\E[N, \boldsymbol{\pi}^{(N)}] {(\mathcal{T}_{k+1}-\mathcal{T}_{k})} \\ \numberthis \label{eq:bound-sum} & \leq \frac{1}{\mass{\boldsymbol{\pi}^{(N)}}} \sum_{k=0}^{\lceil\log_2(\psi'(N))\rceil -1}\frac{2^{k+5}(\sum_{P \in \mathscr{P}_{N}^{(k)}} \cstone_N(P,k)^{-1})}{f_{k}^2}. 
    \end{align*}
\end{linenomath}
Now, recall that according to the definition in~\eqref{eq:rho_k} and just above it, $\delta = \rho \eps$ and $\rho_{k}>\rho$ for all $k$. Therefore we have  
\begin{linenomath}
\begin{align*}
\mathcal{T}_{\lceil\log_{2}(\psi'(N))\rceil} & = \inf\left\{t > 0: \left\langle m \mathbf{1}_{m \geq 2^i}, \bar{\mathbf{L}}^{(N)}_{t} \right\rangle /\mass{\boldsymbol{\pi}^{(N)}} \geq \rho_{i} \text{ for all } i= 0, 1, \ldots, \lceil\log_{2}(\psi'(N))\rceil\right\}
\\ & \geq \inf\left\{t > 0: \left\langle m \mathbf{1}_{m \geq \psi'(N)}, \bar{\mathbf{L}}^{(N)}_{t} \right\rangle /\mass{\boldsymbol{\pi}^{(N)}} \geq \rho\right\}
\\ & \numberthis \label{eq:t-to-tau} \geq \inf\left\{t > 0: \left\langle m \mathbf{1}_{m \geq \psi'(N)}, \bar{\mathbf{L}}^{(N)}_{t} \right\rangle \geq \rho \eps \right\} \stackrel{\eqref{eq:giant-time-coagulation}}{=} \tau_{N}(\psi'(N), \delta),
\end{align*}
\end{linenomath}
almost surely. 

Combining~\eqref{eq:t-to-tau} with~\eqref{eq:bound-sum}, we deduce that 
\begin{linenomath}
\begin{align*} 
\E[N]{\tau_{N}(\psi'(N), \delta)} & \leq \E[N, \boldsymbol{\pi}^{(N)}]{\mathcal{T}_{\lceil\log_{2}(\psi'(N))\rceil}} \\ & \leq \frac{1}{\mass{\boldsymbol{\pi}^{(N)}}} \sum_{k=0}^{\lceil\log_2(\psi'(N))\rceil -1}\frac{2^{k+5}(\sum_{P \in \mathscr{P}_{N}^{(k)}} \cstone_N(P,k)^{-1})}{f_{k}^2} 
\\ \numberthis \label{eq:final-bound} & \leq \frac{1}{\eps} \sum_{k=0}^{\lceil\log_2(\psi'(N))\rceil -1}\frac{2^{k+5}(\sum_{P \in \mathscr{P}_{N}^{(k)}} \cstone_N(P,k)^{-1})}{f_{k}^2}. 
\end{align*}
\end{linenomath}
By Equation~\eqref{eq:summability}, the right-side of~\eqref{eq:final-bound} is a constant $C = C(\eps)$, independent of $\boldsymbol{\pi}^{(N)}$ and $\delta$, from which we deduce~\eqref{eq:thm-main}. In addition, the right-side just above~\eqref{eq:final-bound} tends to zero if $\mass{\boldsymbol{\pi}^{(N)}}\to\infty$, proving the last statement in Item~\ref{item:thme-simp-gel-1} of Theorem~\ref{thm:simp-gel}. 
\end{proof}    
\subsection{Proof of Item~\ref{item:thme-simp-gel-2} of Theorem~\ref{thm:simp-gel}}
\begin{proof}[Proof of Item~\ref{item:thme-simp-gel-2} of Theorem~\ref{thm:simp-gel}]
To prove Item~\ref{item:thme-simp-gel-2} of Theorem~\ref{thm:simp-gel} we apply Item~\ref{item:thme-simp-gel-1} of Theorem~\ref{thm:simp-gel} and anneal over the possible values of $\bar{\mathbf{L}}^{(N)}_{0}$. In this regard, define the event 
\[
\mathcal{E}^{(N)}(\eps, \rho_{0}) : = \left\{\mass{\bar{\mathbf{L}}^{(N)}_{0}} > \eps, \frac{\left\langle m \mathbf{1}_{m \geq 1}, \bar{\mathbf{L}}^{(N)}_{0} \right\rangle}{\mass{\bar{\mathbf{L}}^{(N)}_{0}}} \geq \rho_{0}\right\}. 
\]
 Since the bound in~\eqref{eq:thm-main} is independent of $\boldsymbol{\pi}^{(N)}$ and $N$, see~\eqref{eq:final-bound},
we have 
\[
 \E[N]{\tau_{N}(\psi'(N), \delta) \, \bigg| \,\mathcal{E}^{(N)}(\eps, \rho_{0})} \leq C. 
\]
Using the assumption from~\eqref{eq:liminf-prob-gel}, set $p_0:= \limsup \Prob[N]{\mathcal{E}^{(N)}(\eps, \rho_{0})}>0$. Then, there exists a subsequence $(N_j)_{j\in\mathbb{N}}$ such that $\lim_{j\to\infty}\Prob[N_j]{\mathcal{E}^{(N_{j})}(\eps, \rho_{0})} = p_0$. In particular,  for $N_j$ sufficiently large, \[\Prob[N_j]{\mathcal{E}^{(N_{j})}(\eps, \rho_{0})}>0.\] 
By  Markov's inequality,
\[
\Prob[N_j]{\tau_{N_{j}}(\psi'(N_j), \delta) \leq 2C \, \bigg| \, \mathcal{E}^{(N_{j})}(\eps, \rho_{0})} \geq \frac{1}{2}.
\]
We thus get
\begin{align*}
&\limsup_{N \to \infty} \Prob[N]{\tau_{N}(\psi'(N), \delta) \leq 2C}\geq 
\lim_{j\to\infty} \Prob[N_j]{\tau_{N_{j}}(\psi'(N_j), \delta) \leq 2C} \\
& \hspace{2cm} \geq\lim_{j\to\infty} \Prob[N_j]{\left\{\tau_{N_{j}}(\psi'(N_j), \delta) \leq 2C \right\} \cap \mathcal{E}^{(N_{j})}(\eps, \rho_{0})}\\
& \hspace{2cm} =\lim_{j\to\infty}\Prob[N_j]{\tau_{N_{j}}(\psi'(N_j), \delta) \leq 2C \, \bigg| \, \mathcal{E}^{(N_{j})}(\eps, \rho_{0})}\Prob[N_j]{\mathcal{E}^{(N_{j})}(\eps, \rho_{0})}\geq
\frac{p_0}{2}. 
\end{align*}

We deduce that stochastic gelation occurs according to Definition~\ref{defn:stoch-gelation} (with $T^{\psi', \delta}_g\leq 2C$). 
\end{proof}
\subsection{Proof of Corollary~\ref{cor:classical-gelation}}
We finish this section with the proof of Corollary~\ref{cor:classical-gelation}. 

\begin{proof}[Proof of Corollary~\ref{cor:classical-gelation}]
    To prove the first statement of the corollary we apply Item~\ref{item:thme-simp-gel-2} of Theorem~\ref{thm:simp-gel}. In order to do so, we need to show that Assumption~\ref{ass:gelation} is satisfied. Recall that in this setting $E=(0, \infty)$ and $m(x)=x$ for all $x\in (0, \infty)$. Therefore, in Assumption~\ref{ass:gelation}, for every $j$
    we can take trivial single-set partition of $[2^j,2^{j+1})$, so that $\xi(N) \equiv 1$. Since each of the partitions $\mathscr{P}^{(j)}_{N}$ consist of the single set $[2^j,2^{j+1})$, and we assumed $c'(j) = \inf_{x,y \in [2^{j}, 2^{j+1})} \bar{K}(x,y) > 0$, we also have, for all $N\in\mathbb{N}$
\begin{align}
&\cstone_N(P,j)\equiv \cstone_N([2^j,2^{j+1}),j)\equiv \cstone(j) >0,  \qquad \forall j\in\mathbb{N}_0, \label{cstone}
\end{align}
where $\cstone_N(P,j)$ are as defined in Item~\ref{item:partition_lower_bound} of Assumption~\ref{ass:gelation}. Now, for Item~\ref{item:condition_diagonal}, Equation~\eqref{eq:summability} reduces to showing that for some sequence $(f_{j})_{j \in \mathbb{N}_{0}} \in (0, \infty)^{\mathbb{N}_{0}}$ with $\sum_{j=0}^{\infty} f_{j} < \infty$, we have
\[
\sum_{j=0}^{\infty} \frac{2^{j}}{f_{j}^2 \cstone(j)} < \infty, 
\]
where we note that the sum may be taken to infinity because in this case the term $\phi_{N}$ is independent of $N$. However, by Remark~\ref{equiv}, this is equivalent to~\eqref{eq:class-gel-sum-cond}. Finally, for Item~\ref{item:condition_pidgeon} of Assumption~\ref{ass:gelation}, since $\xi(N) \equiv 1$, we clearly have $\lim_{N \to \infty} \frac{\xi(N)}{N} = 0$. 
This concludes the proof of the first statement of the corollary.

For the second assertion, note that in either of the two cases, we have \[\cstone(j)= \inf_{x,y \in [2^{j}, 2^{j+1})}\bar{K}(x,y) >0 \quad \text{ for all } \quad j \in \mathbb{N}_{0}.\] Hence, we only need to show that under each assumption,~\eqref{eq:class-gel-sum-cond} is satisfied. Indeed, 
\begin{enumerate}
    \item Under Item~\ref{item:no_space1}, set $\kappa^{*} := \inf_{i \in [1, 2]} \bar{K}(1,i)>0$. By the homogeneity assumption there exists $j_0 \in \mathbb{N}$ such that, for all $j \geq j_0$ we have (assuming without loss of generality $x<y$), $\bar{K}(x, y)= x^{\gamma}\bar{K}\left(1,\frac{y}{x}\right) \geq \kappa^{*} 2^{\gamma j}=:c'(j)$ whenever $2^{j} \leq x, y < 2^{j+1}$.  
    This implies that 
    \[
    \sum_{j=0}^{\infty} \left(\frac{2^{j}}{\cstone(j)} \right)^{\frac 13} = \frac{1}{(\kappa^{*})^{1/3}} \sum_{j=0}^{\infty} 2^{(1-\gamma)/3} < \infty,
    \]
    whenever $\gamma > 1$. 
    \item Under Item~\ref{item:no_space2}, by assumption  
    for all $j \geq 0$ we have $\bar{K}(x, y) \geq 1+2^{j} (j\log{2})^{3 + \epsilon}=:c'(j)$ for $ 2^{j}\leq x,y <2^{j+1}$. 
    This implies that 
    \[
    \sum_{j=0}^{\infty} \left(\frac{2^{j}}{\cstone(j)} \right)^{\frac 13} \leq 1+ \frac{1}{(\log{2})^{1 +  \epsilon/3}}\sum_{j=1}^{\infty} \frac{1}{j^{1 + \epsilon/3}} < \infty.
    \]
    \end{enumerate}
\end{proof}
\begin{example}\label{ex_final}
As mentioned in Remark~\ref{rem_cor}, in the setting of  Corollary~\ref{cor:classical-gelation}, we can quantify the sizes of the large clusters contributing to gelation. 
Indeed, the proof of Item~\ref{item:thme-simp-gel-1} of Theorem \ref{thm:simp-gel} can be retraced with the choice of  $f_k:=\frac{1}{k^{1+\epsilon}}$ (recall the assumptions introduced around Equation~\eqref{eq:rho_k}). In particular condition~\eqref{eq:ass_altern1} reduces to check whether 
\[\frac{N}{4}\geq \psi'(N) (\lceil \log_2(\psi'(N))\rceil)^{1+\epsilon} \qquad \forall N \in \mathbb{N}
\]
(recall that $\xi(N)\equiv 1$).
This can be achieved by taking for instance any $\psi'(N)\leq N^{b}$ for any $b\in(0,1)$, and $\epsilon$ sufficiently small.  
\end{example}

{\bf Acknowledgements.}
We would like to thank Wolfgang K\"onig, and Robert Patterson for some helpful discussions. 
This research has been partially funded by the
Deutsche Forschungsgemeinschaft (DFG) (project number 443759178) through grant SPP2265 ``Random Geometric Systems'', Project P01: `Spatial Coagulation and Gelation'. The authors acknowledge the support from Dipartimento di Matematica e Informatica ``Ulisse Dini'', University of Florence.  LA acknowledges partial financial support by~``Indam-GNAMPA''~Project~CUP\_E53C22001930001 and by the Italian Ministry of University and Research (MUR) via PRIN 2022– Project Title ConStRAINeD – CUP-2022XRWY7W. We also would like to thank an anonymous reviewer for their detailed, constructive feedback, which has greatly improved the manuscript, in particular leading to a number of simplifications of Theorem~\ref{thm:simp-gel}.

\bibliographystyle{abbrv}
\bibliography{ref}

\end{document}